\documentclass[11pt,titlepage]{article}
\usepackage{graphicx}
\usepackage{amsmath}
\usepackage{amsfonts}
\usepackage{amsthm}
\usepackage{amssymb}
\usepackage{color}

\usepackage{verbatim} 
\usepackage{url}
\usepackage[margin=1in]{geometry}
\usepackage{setspace}
\newtheorem{theorem}{Theorem}[section]
\newtheorem{lemma}[theorem]{Lemma}
\newtheorem{corollary}[theorem]{Corollary}

\newtheorem{conjecture}[theorem]{Conjecture}

\begin{document}

\title{Modified Stern-Brocot Sequences}

\author{Dhroova Aiylam \\ Mass Academy of Math and Science}

\maketitle

\begin{abstract}

	We discuss and prove several of the properties of the Stern-Brocot tree, including in particular the cross-determinant, before proposing a variant to the tree. In this variant, we allow for arbitrary choice of starting terms. We prove that regardless of the two starting terms every rational number between them appears in the tree.
	
\end{abstract}

\section{Introduction}

The Stern–-Brocot tree was discovered independently by Moritz Stern \cite{Stern} in 1858 and Achille Brocot $\cite{Brocot}$ in 1861. It was originally used by Brocot to design gear systems with a gear ratio close to some desired inexact value (like the number of days in a year) by finding a ratio of smooth numbers (numbers that decompose into small prime factors) near that value. Since smooth numbers factor into small primes, several small gears could be connected in sequence to generate an effective ratio of the product of their teeth; This would make a gear train of reasonable size possible, but would minimize its error \cite{Clocks}.

The tree begins with the terms $\frac{0}{1}$ and $\frac{1}{1}$. In each subsequent row, all terms are copied and between every pair of neighboring terms $\frac{a}{b}$ and $\frac{c}{d}$ the mediant fraction $\frac{a + c}{b + d}$ is put in lowest terms and inserted. This process is repeated ad infinitum; the result is the Stern-Brocot tree.

What Brocot had inadvertently done was develop a computationally easy way to find the best rational approximation to a fraction with a smaller denominator. It was quite well known that continued fractions could be used for the same purpose \cite{CFrac}, which sparked an interest in the connection between the two. Indeed, it was later discovered that the mediant could also be expressed as an operation on the continued fraction expansion of two fractions, whose continued fractions were already very close by virtue of their proximity. In fact, continued fractions provided a way to determine with some certainty exactly where a particular fraction would appear in the tree \cite{CTN}. Retracing the tree upward would then give a series of progressively worse rational approximation with decreasing denominator.

There are many other topics, albeit less well-known, that are related to the Stern-Brocot tree \cite{SB}. Farey Sequences, ordered lists of the rationals between $0$ and $1$ with denominator smaller than $n$, can be obtained by discarding fractions with denominator more than $n$ from the corresponding row of the Stern-Brocot tree \cite{CTN}. The Calkin-Wilf tree is another binary tree generated from a mediant-like procedure. Finally, 
the radii of Ford circles vary inversely with the square of the corresponding term in the left half of the Stern-Brocot Tree \cite{CTN}. Below is a visual representation:

\begin{figure}
	\centering
		\includegraphics[scale=0.80]{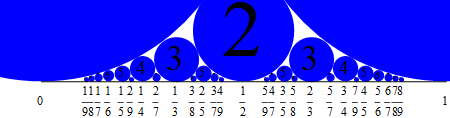}
	\label{fig:Ford}	
\end{figure}

In this paper, we begin by discussing the Stern-Brocot tree and proving several of its properties. We mention the symmetry of the tree, certain algebraic relations its elements satisfy, and whether its terms reduce. We then introduce the notion of the cross-determinant and analyze its role in the reduction of fraction, en route to a proof of the startling fact that every rational number between $0$ and $1$ appears in the Stern-Brocot tree.

We next present a variant to the original Stern-Brocot tree. We consider starting with terms other than $\frac{0}{1}$ and $\frac{1}{1}$, and ask ourselves which properties of the original tree extend to this one.  In particular, we prove that once again every rational number between the two starting terms appears in the tree. We do this first for special types of cross-determinant in Theorems 3.1, 3.2, and 3.3. In Theorem 3.7, we establish the result in general. As part of this proof, we develop the important idea of tree equivalence.

\section{Notation and Definitions}

In number theory, the Stern-Brocot tree is an infinite complete binary tree in which the vertices correspond precisely to the positive rational numbers. We define the Stern-Brocot tree in terms of Stern-Brocot sequences. The $0^{th}$ row of the tree, also the $0^{th}$ Stern-Brocot sequence, is $\frac{0}{1}, \frac{1}{1}, \frac{1}{0}$, which we denote by $SB_0$. In general, $SB_i$ is the $i^{th}$ Stern-Brocot sequence. Each successive sequence is formed by copying all terms from the previous sequence, inserting between every pair of consecutive fractions in the previous sequence their mediant, and reducing any fractions not already in simplest terms.

The $\emph{mediant}$ of two reduced fractions $\frac{a}{b}, \frac{c}{d}$ is the reduced fraction $\frac{a + c}{b + d}$.

The first few sequences (also the first few rows of the tree) are $$SB_0 = \frac{0}{1}, \; \; \; \frac{1}{1}, \; \; \; \frac{1}{0}$$  $$SB_1 = \frac{0}{1}, \; \; \;{\bf \frac{1}{2}}, \; \; \; \frac{1}{1}, \; \; \; {\bf \frac{2}{1}}, \; \; \; \frac{1}{0}$$  $$SB_2 = \frac{0}{1}, \; \; \; {\bf \frac{1}{3}}, \; \; \; \frac{1}{2}, \; \; \; {\bf \frac{2}{3}}, \; \; \; \frac{1}{1}, \; \; \; {\bf \frac{3}{2}}, \; \; \; \frac{2}{1}, \; \; \; {\bf \frac{3}{1}}, \; \; \; \frac{1}{0}.$$

In bold are the mediant fractions that have been inserted.

It is quite clear that these sequences are reciprocally symmetric with respect to their center, $\frac{1}{1}$; that is, the $j^{th}$ 
term counted from the left is the reciprocal of the $j^{th}$ term counted from the right. In light of this, we will consider only the left half of these sequences, between $0$ and $1$ inclusive, which we will call Stern-Brocot half-sequences.

We begin with a definition and some small lemmas.

Let the $\emph{cross-determinant}$ of two consecutive reduced fractions $\frac{a}{b}, \frac{c}{d}$ in a Stern-Brocot half-sequence equal $bc - ad$. The following lemma is well-known $\cite{CTN}$; we provide a proof because it illustrates a method used later.

\begin{lemma}
For any two consecutive fractions $a/b$ and $c/d$ in a sequence, the cross-determinant of the pair equals 1.
\end{lemma}

\begin{proof}
We prove this by induction. For the zeroeth half-sequence $\frac{0}{1}, \frac{1}{1}$ the lemma holds. Suppose that for any two consecutive elements in the $i^{th}$ half-sequence and for all $i \le k$ the lemma holds. Let $a/b$ and $c/d$ be two consecutive fractions in the $k^{th}$ half-sequence. Their mediant is equal to $\frac{a + c}{b + d}$ which can be written as $$\frac{\frac{a + c}{g}}{\frac{b + d}{g}}$$ in reduced form, where $g = \gcd({a + c, b + d})$. Then the determinant $bc - ad$ can be written as $(a + c)b - (b + d)a$, which is divisible by $g$. But since $bc - ad = 1$ by our inductive hypothesis, $g = 1$. Thus $\frac{a + c}{b + d}$ is reduced.

Then any two consecutive entries in row $k + 1$ are either $\frac{a}{b}, \frac{a + c}{b + d}$ or $\frac{a + c}{b + d}, \frac{c}{d}$ for $\frac{a}{b}$ and $\frac{c}{d}$ consecutive in the previous sequence. The determinant for the first pair is $$b(a + c) - a(b + d) = bc - ad = 1,$$ and for the second pair it is $$(b + d)c - (a + c)d = bc - ad = 1,$$ Hence the determinant of any two consecutively occuring fractions is $1$. 
\end{proof}

As a part of this proof, we have established the following corollary:

\begin{corollary}
The mediant fractions in every row never need to be reduced.
\end{corollary}

The next two lemmas are quite simple $\cite{SB}$; their proofs are left as an exercise for the reader.

\begin{lemma}
There are exactly $2^i + 1$ elements in the $i^{th}$ Stern-Brocot half-sequence.
\end{lemma}

\begin{lemma}
Stern-Brocot half-sequences are algebraically symmetric; that is, opposite entries add to $1$.
\end{lemma}

Let us denote by $SB_i[j]$ the $j$-th element in the $i$-th half-sequence. Also, let us denote by $N_i$ and $D_i$, the ordered set of numerators and denominators of $SB_i$. For example, $$N_2 = \{0, 1, 1, 2, 1\},$$ $$D_2 = \{1, 3, 2, 3, 1\}.$$ Notice that these sets have the same size --- $2^i + 1$ in general. We will therefore take the sum $A + B$ of two ordered sets of equal size to mean the ordered set of the same size where each element is the sum of the corresponding elements in $A$ and $B$.

Let $F$ be the function that takes any set of fractions to the set produced by copying each fraction and inserting between consecutive ones their mediant. Specifically, $F(SB_i) = SB_{i + 1}$. Next let $G$ be the function that takes a set of whole numbers to the set formed by inserting between each pair of numbers their sum. For example, $$\{1, 2, 3, 4\} \rightarrow \{1, 3, 2, 5, 3, 7, 4\}$$ under application of $G$. Finally, we take $X[a, b]$ to denote the $a^{th}$ through $b^{th}$, inclusive, elements of set $X$.

The following lemma allows us to recursively describe the numerators and denominators of successive sequences.

\begin{lemma}
$N_{i + 1}[0, 2^i] = N_i$ and $D_{i + 1}[0, 2^i] = N_i + D_i.$
\end{lemma}

\begin{proof}
We prove by induction on $i$.

The result is easily checked for $i = 0$. Now suppose the result holds for all $i \le k$.

First we need $$N_{k + 2}[0, 2^{k + 1}] = N_{k + 1},$$

We have 
$$SB_{k + 2}[0, 2^{k + 1}] = F(SB_{k + 1})[0, 2^{k + 1}] = F(SB_{k + 1}[0, 2^{k}]).$$

We can therefore equate the set of numerators of the left and right hand side, so

$$N_{k + 2}[0, 2^{k + 1}] = G(N_{k + 1}[0, 2^{k}])$$

since fractions never reduce. But by the induction hypothesis, 
$$N_{k + 1}[0, 2^k] = N_k$$ meaning 
$$G(N_{k + 1}[0, 2^{k}]) = G(N_k) = N_{k + 1}.$$
where the last step follows again from the fact that fractions never reduce.

Additionally, we need $$D_{k + 2}[0, 2^{k + 1}] = N_{k + 1} + D_{k + 1}$$
Again, $$SB_{k + 2}[0, 2^{k + 1}] = F(SB_{k + 1})[0, 2^{k}] = G(SB_{k + 1}[0, 2^{k}]).$$ Equating denominators we have
$$D_{k + 2}[0, 2^{k + 1}] = G(D_{k + 1}[0, 2^k])$$ since fractions never reduce. By the induction hypothesis, $$G(D_{k + 1}[0, 2^k]) = G(D_k + N_k),$$ and since $G$ is additive, $$ G(D_k + N_k) = G(D_k) + G(N_k).$$ Once more, fractions never reduce, so $$ G(D_k) + G(N_k) = D_{k + 1} + N_{k + 1},$$ as desired.

\end{proof}

We can apply these lemmas together to prove a fascinating result; while the theorem is well-known, the proof is original to the best of our knowledge.

\begin{theorem}
All rational numbers between $0$ and $1$ appear in some (and, of course, every subsequent) Stern-Brocot half-sequence.
\end{theorem}

\begin{proof}
Let us induct on the denominator $d$. When $d = 1$, we need $\frac{0}{1}$ and $\frac{1}{1}$ to appear in some half-sequence, which they clearly do. Suppose the result holds for all $d \le k$. Then when $d = k + 1$, we need all fractions of the form $\frac{a}{k + 1}$ where $1 \le a \le k + 1$ and $\gcd{(a, k + 1)} = 1$ to appear in some half-sequence. If $\frac{a}{k + 1} < \frac{1}{2}$, our previous lemma guarantees that this fraction exists if $\frac{a}{(k + 1) - a}$ appears in the previous half-sequence. But the denominator of this fraction is strictly less than $k + 1$, so we are done by the inductive hypothesis. Otherwise, if $\frac{a}{k + 1} > \frac{1}{2}$, the algebraic symmetry of Stern-Brocot sequences means this fraction appears if and only if $\frac{(k + 1) - a}{k + 1} < \frac{1}{2}$ and so by the same argument as above, we are done.
\end{proof}

Of course, the reciprocal property of the complete Stern-Brocot sequence means it contains every non-negative rational number.

\section{Arbitrary Starting Terms}

One variant of the Stern-Brocot tree that arises quite naturally comes from varying the two starting terms; that is, beginning instead with any pair of non-negative rational numbers. The process of inserting mediants is exactly the same: the mediant fraction $\frac{a + c}{b + d}$ is reduced and inserted between the consecutive fractions $\frac{a}{b}$ and $\frac{c}{d}$. Since the cross-determinant is no longer necessarily $1$, the reduction step is significant.

Example.

	\[ \frac{2}{5} \; \; \; \frac{5}{11}\]  \[\frac{2}{5} \; \; \; \frac{7}{16} \; \; \; \frac{5}{11}\] 
	\[\frac{2}{5} \; \; \; {\bf \frac{3}{7}} \; \; \;  \frac{7}{16} \; \; \;  {\bf \frac{4}{9}} \; \; \; \frac{5}{11}\] 
	\[ \frac{2}{5} \; \; \; \frac{5}{12} \; \; \; \frac{3}{7} \; \; \; \frac{10}{23} \; \; \;  \frac{7}{16} \; \; \; \frac{11}{25} \; \; \; \frac{4}{9} \; \; \; \frac{9}{20} \; \; \; \frac{5}{11}\]

	Notice that the fractions in bold have been reduced.

We shall investigate how these generalized sequences behave and whether they exhibit properties similar to the original Stern-Brocot sequences. Of particular interest to us is whether each rational number between the two starting terms appears somewhere in the sequence, a strong claim which the original sequence satisfies. We are also interested in the cross-determinant, as it is central to the behavior of the sequence because of its role in determining where and when the terms of the sequence must be reduced. We now present three results, characterized by the value of the cross-determinant.

Let $S_n(a, b) = F^{(n)}(a, b)$ be the $n^{th}$ sequence formed by repeatedly inserting mediants between consecutive fractions. Also denote by $T(a, b)$ the tree formed by all the $S_i(a, b)$. For example, $T(2, 3)$ is $$\frac{2}{1} \; \; \; \frac{3}{1}$$ $$\frac{2}{1} \; \; \; \frac{5}{2} \; \; \; \frac{3}{1}$$  $$\frac{2}{1} \; \; \; \frac{7}{3} \; \; \; \frac{5}{2} \; \; \; \frac{8}{3} \; \; \; \frac{3}{1}$$ $$\dots$$

\begin{theorem}
If fractions $\frac{a}{b}$ and $\frac{c}{d}$ satisfy $bc - ad = 1$, every rational number in the interval $[\frac{a}{b}, \frac{c}{d}]$ appears in $T(a, b)$.
\end{theorem}

\begin{proof}
We know the result holds when $\frac{a}{b} = \frac{0}{1}$ and $\frac{c}{d} = \frac{1}{1}$. Every number in this range can be written as $\frac{z}{w + z}$ for some choice of $w, z$. But this is just $$\frac{0w + 1z}{1w + 1z}$$ which means we can obtain any combination of $2$ weights that describe how many left and right mediants have been taken.

If $bc - ad = 1$ for some $a, b, c, d$, reduction of fractions never takes place, meaning we only need to show that any number $\frac{x}{y}$ with $\frac{a}{b} \le \frac{x}{y} \le \frac{c}{d}$ can be written as $$\frac{aw + cz}{bw + dz}$$ for appropriate choice of $w, z$, since under the transformation $0 \rightarrow a, 1 \rightarrow b, 1 \rightarrow c, 1 \rightarrow d$, we can reduce the problem to the appearance of a particular fraction in $T(\frac{0}{1}, \frac{1}{1})$ which we know happens.

We want $$\frac{aw + cz}{bw + dz} = \frac{x}{y},$$ or equivalently,

$$(ay)w + (cy)z = (bx)w + (dx)z,$$ $$(bx - ay)w = (cy - dx)z$$ $$(bx - ay)(w + z) = [(bx - ay) + (cy - dx)]z,$$ $$\frac{z}{z + w} = \frac{(bx - ay)}{(bx - ay) + (cy - dx)},$$ meaning we can take $w = cy - dx$ and $z = bx - ay$ which are both positive integers. Then the fraction $\frac{x}{y}$ appears in $T(\frac{a}{b}, \frac{c}{d})$.

\end{proof} 

For two special types of cross-determinant, Theorem 3.1 alone is sufficient to prove that every rational number in between the two starting terms is contained in the tree.

\begin{theorem}
If the cross-determinant of $\frac{a}{b}$ and $\frac{c}{d}$ is a power of two, every rational number in the interval $[\frac{a}{b}, \frac{c}{d}]$ appears in $T(\frac{a}{b}, \frac{c}{d})$. Furthermore, for all $i \ge L$ for some $L$ sufficiently large each pair of consecutive elements in $S(\frac{a}{b}, \frac{c}{d}, i)$ has cross-determinant $1$. 
\end{theorem}

\begin{proof}
We prove by induction. When $bc - ad = 2^0 = 1$, this is simply Theorem 3.1 . Suppose the result holds for $2^i$ where $i \le k$, so we want to show the result for $2^{k + 1}$. Consider the parity of $a, b, c, d$. Since $\frac{a}{b}$ and $\frac{c}{d}$ are in lowest terms, $a, b$ cannot both be even and $c, d$ cannot both be even. Also, $bc$ and $ad$ are either both odd or both even (since their difference is a power of $2$). If they are both odd, $a, b, c, d$ are all odd. If they are both even, either $a, c$ are even and $b, d$ are odd or $a, c$ are odd and $b, d$ are even. In either case, the numerator and denominator of the resulting mediant fraction are both even, meaning it must be reduced. Since $bc - ad = 2^{k + 1},$ the factor it is reduced by must itself be a power of two, so we have in the subsequent sequence $$\frac{a}{b} \; \; \; \frac{m}{n} \; \; \; \frac{c}{d},$$ where $bm - an = cn - dm = 2^j$ for $j \le k$. Then by our inductive hypothesis, all rational numbers in the intervals $[\frac{a}{b}, \frac{m}{n}]$ and $[\frac{m}{n}, \frac{c}{d}]$ must appear, meaning every number in their union, $[\frac{a}{b}, \frac{c}{d}]$, must appear, as desired. 

Also by the inductive hypothesis, the set of cross-determinants in each of the two smaller intervals $[\frac{a}{b}, \frac{m}{n}]$ and $[\frac{m}{n}, \frac{c}{d}]$ will consist of only $1$s after finitely many steps, so after finitely many steps their union, the set of cross-determinants of the pair $\frac{a}{b}, \frac{c}{d}$, will consist of only $1$'s and so we are done.
\end{proof}

\begin{theorem}
If the cross-determinant of $\frac{a}{b}$ and $\frac{c}{d}$ is of the form $2^m3^n$, every rational number in the interval $[\frac{a}{b}, \frac{c}{d}]$ appears in the sequence. Furthermore, after finitely many rows, each pair of consecutive elements in all subsequent rows has cross-determinant $1$.
\end{theorem}

\begin{proof}
By Theorem 3.2, suffice it to consider the union of finitely many intervals $$[x_1, x_2] \cup [x_2, x_3] \cup \dots \cup [x_{t - 1}, x_t]$$ where the cross-determinant of $x_i$ and $x_{i + 1}$ is a power of $3$. We claim that if $bc - ad = 3^j$, the cross-determinants of consecutive terms taken from the set $$\frac{a}{b} \; \; \frac{2a + c}{2b + d} \; \; \frac{a + c}{b + d} \; \; \frac{a + 2c}{b + 2d} \; \; \frac{c}{d}$$ are at most $3^{j - 1}$. To prove this we do casework on the different possible values of $a, b, c, d \pmod{3}$ using the fact that $3 | bc - ad$ means $bc \equiv ad \pmod{3}$.

If $bc \equiv ad \equiv 0 \pmod{3}$, then $3|b$ or $3|c$ and $3|a$ or $3|d$. Since $(a, b) = (c, d) = 1$, this means either $3|b$ and $3|d$ or $3|a$ and $3|c$. Without loss of generality suppose $3|a$ and $3|b$. Then since $3 \nmid b$ and $3 \nmid d$ either $b \equiv -d \pmod{3}$, in which case $\frac{a + b}{c + d}$ will be reduced, or $b \equiv -2d \pmod{3}$, meaning the two fractions $\frac{2a + c}{2b + d}$ and $\frac{a + 2c}{b + 2d}$ will be reduced. In either case, the maximum possible cross-determinant of any pair of consecutive fractions in the next sequence is $3^{j - 1}$.

If $bc \equiv ad \equiv 1 \pmod{3}$, then $(b, c), (a, d) \pmod{3} \in \{(1, 1), (2, 2)\}$. If $(b, c)$ and $(a, d)$ are the same (that is, both either $(1, 1)$ or $(2, 2)$), then $3 | 2a + c$, $3 | 2b + d$, $3 | a + 2c$, $3| b + 2d$ so again both fractions $\frac{2a + c}{2b + d}$ and $\frac{a + 2c}{b + 2d}$ are reducible. If instead one of $(a, b), (c, d)$ is $(1, 1)$ and the other is $(2, 2)$, then $3 | a + c$ and $3 | b + d$, so the fraction $\frac{a + c}{b + d}$ is reducible and again the maximum possible cross-determinant of any pair of consecutive fractions is $3^{j - 1}$.

Finally, if $bc \equiv ad \equiv 2 \pmod{3}$, $(b, c), (a, d) \pmod{3} \in \{(2, 1), (1, 2)\}$. If they are the same, $3 | 2a + c$, $3 | 2b + d$, $3 | a + 2c$, and $3 | b + 2d$ so both fractions $\frac{2a + c}{2b + d}$ and $\frac{a + 2c}{b + 2d}$ are reducible. Otherwise, $3 | a + c$ and $3 | b + d$ so $\frac{a + c}{b + d}$ is reducible and so all cases are covered.

In any case, after at most $2j$ iterations all pairs of consecutive terms have cross-determinant $1$, meaning we can apply Theorem 3.1 to finish.

\end{proof}

We would now like to generalize this result to all possible values of the cross-determinant; that is, to show that regardless of the value of $bc - ad$ the tree $T(\frac{a}{b}, \frac{c}{d})$ contains all rational numbers in $[\frac{a}{b}, \frac{c}{d}]$. To do this, we will first introduce the notion of \emph{corresponding elements} and \emph{equivalent trees}. 

Let $e_1$ be an element of some tree $T_1$ such that it occupies position $p$ in row $r$ of $T_1$. For any other tree $T_2$, we will call $e_1$ and $e_2 \in T_2$ \emph{corresponding elements} if and only if $e_2$ occupies position $p$ in row $r$ of $T_2$.

Then given two trees, we say they are \emph{equivalent} if and only if all pairs of corresponding elements are reduced by exactly the same factor. Equivalent trees are very closely related in structure. In fact, 

\begin{theorem}
Let $T_1 = T(\frac{a_1}{b_1}, \frac{c_1}{d_1})$ and $T_2 = T(\frac{a_2}{b_2}, \frac{c_2}{d_2})$ be two equivalent trees. If $e_1 = \frac{p_1}{q_1} \in T_1$ and $e_2 = \frac{p_2}{q_2} \in T_2$ are corresponding elements, then $e_1$ and $e_2$ are the same weighted combination of the initial terms in their respective trees. That is, let $(x, y, g)$ be the unique triple of positive integers satisfying $\gcd({x, y})$ = 1, $\frac{p_1}{q_1} = \frac{a_1x + c_1y}{b_1x + d_1y}$, and $g = \gcd({a_1x + c_1y, b_1x + d_1y})$. Then $\frac{p_2}{q_2} = \frac{a_2x + c_2y}{b_2x + d_2y}$ and $g = \gcd({a_2x + c_2y, b_2x + d_2y})$.
\end{theorem}

We saw a special case of this concept briefly in the proof of Theorem $3.1$. We present a formal proof of the general statement now.

\begin{proof}
We will prove by induction on the row number $r$. When $r = 0$, the conclusion is obvious. Then suffice it to show, given consecutive fractions $m_1, n_1$ in a row of $T_1$ and the corresponding $m_2, n_2$ in $T_2$ which satisfy the statement of the theorem, that their mediant fractions do as well. We can write $$m_1 = \frac{\frac{w_1a_1 + z_1c_1}{g_1}}{\frac{x_1b_1 + y_1d_1}{g_1}}$$ and $$n_1 = \frac{\frac{w_2a_1 + z_2c_1}{g_2}}{\frac{x_2b_1 + y_1d_1}{g_2}}$$ for the appropriate $w_1, z_1, w_2, z_2, g_1, g_2$ so that each fraction is now written in lowest terms. Taking the mediant, we arrive at $$s_1 = \frac{\frac{(g_2w_1 + g_1w_2)a_1 + (g_2z_1 + g_1z_2)c_1}{g_1g_2}}{\frac{(g_2w_1 + g_1w_2)b_1 + (g_2z_1 + g_1z_2)d_1}{g_1g_2}}$$

The numerator and denominator of $s_1$ are not necessarily coprime; to account we let $g$ be the $\gcd$ of the numerator and denominator whence we can write the numerator and denominator of $s_1$ exactly as $$\frac{(g_2w_1 + g_1w_2)a_1 + (g_2z_1 + g_1z_2)c_1}{gg_1g_2}$$ and $$\frac{(g_2w_1 + g_1w_2)a_1 + (g_2z_1 + g_1z_2)c_1}{gg_1g_2}$$ respectively. 

For our other tree, analogous algebra gives as the mediant of $m_2$ and $n_2$ $$s_2 = \frac{\frac{(g_2w_1 + g_1w_2)a_2 + (g_2z_1 + g_1z_2)c_2}{g_1g_2}}{\frac{(g_2w_1 + g_1w_2)b_2 + (g_2z_1 + g_1z_2)d_2}{g_1g_2}}$$ Once more, we must divide to account for the fact that the numerator and denominator of $s_2$ are not necessarily coprime. However, $T_1$ and $T_2$ are equivalent, so the factor by which they are reduced is the same --- $g$. Then 
$$\frac{(g_2w_1 + g_1w_2)a_2 + (g_2z_1 + g_1z_2)c_2}{gg_1g_2}$$ and $$\frac{(g_2w_1 + g_1w_2)b_2 + (g_2z_1 + g_1z_2)d_2}{gg_1g_2}$$ are the numerator and denominator of $s_2$.

Notice now that $s_1$ and $s_2$ have equal weights $(g_2w_1+ g_1w_2, g_2z_1 + g_1z_2)$ and that $$\gcd((g_2w_1 + g_1w_2)a_1 + (g_2z_1 + g_1z_2)c_1, (g_2w_1 + g_1w_2)b_1 + (g_2z_1 + g_1z_2)d_1) =$$ $$\gcd((g_2w_1 + g_1w_2)a_2 + (g_2z_1 + g_1z_2)c_2, (g_2w_1 + g_1w_2)b_2 + (g_2z_1 + g_1z_2)d_2) = gg_1g_2$$ so by induction the result holds.
\end{proof}

This almost immediately gives the following lemma:

\begin{lemma}
If $T_1 = T(\frac{a_1}{b_1}, \frac{c_1}{d_1})$ and $T_2 = T(\frac{a_2}{b_2}, \frac{c_2}{d_2})$ are equivalent trees and $T_2$ contains all rational numbers in the interval $[\frac{a_2}{b_2}, \frac{c_2}{d_2}]$, then $T_1$ contains all rational numbers in the interval $[\frac{a_1}{b_1}, \frac{c_1}{d_1}]$.
\end{lemma}

\begin{proof}
$T_2$ contains all rational numbers in the interval $[\frac{a_2}{b_2}, \frac{c_2}{d_2}]$ if and only if all possible weights $(x, y)$ are attainable. This is because every rational number in this interval can be written as a weighted combination of $\frac{a}{b}, \frac{c}{d}$, which means if some pair of weights $(x, y)$ is not attainable, the corresponding fraction does not appear. But since $T_1$ and $T_2$ are equivalent, the set of weights attainable in $T_1$ is exactly the set of weights attainable in $T_2$. Since all possible weights are attainable in $T_2$, they are all attainable in $T_1$, so $T_1$ contains all rational numbers in $[\frac{a_1}{b_1}, \frac{c_1}{d_1}]$.
\end{proof}

Now that we can indirectly show that a tree $T(\frac{a}{b}, \frac{c}{d})$ contains all rational numbers in the interval $[\frac{a}{b}, \frac{c}{d}]$, we are motivated to establish equivalence between the general tree $T(\frac{a}{b}, \frac{c}{d})$ and some particularly malleable one.

\begin{theorem}
For any tree $T(\frac{a}{b}, \frac{c}{d})$, there exists a positive integer $v$ such that $T(\frac{a}{b}, \frac{c}{d})$ is equivalent to the tree $T(\frac{0}{1}, \frac{D}{v})$, where $D = bc - ad$ is the cross determinant of the pair $\frac{a}{b}$, $\frac{c}{d}$.
\end{theorem}

\begin{proof}
Suppose there existed a positive integer $V$ such that $\gcd(ax + cy, bx + dy) = \gcd(Dy, x + Vy)$ for all $x, y$. We claim it would follow that $T_1 = T(\frac{a}{b}, \frac{c}{d})$ and $T_2 = T(\frac{0}{1}, \frac{D}{V})$ are equivalent.

The first row of $T_1$ and $T_2$ are of course equivalent, so initially, corresponding entries have the same weights. When reduction of fractions with the same weights takes place, we can now be sure it is by exactly the same factor so corresponding reduced fractions also have the same weights. These two assumptions are exactly the inductive hypothesis of the proof of Theorem 3.4, so by an identical argument $T_1$ and $T_2$ are equivalent.

It remains only to show that some such $V$ exists. Let the prime factorization of $D$ be $p_1^{e_1}p_2^{e_2} \dots p_k^{e_k}$. We know that if two fractions in a tree with determinant $D$ reduce by some factor $g$, then $g | D$.  Then if we can show that there exists some $v$ such that for all $p_i$, $\min(v_{p_i}(ax + cy), v_{p_i}(bx + dy)) = \min(v_{p_i}(Dy), v_{p_i}(x + Vy)) = v_{p_i}(x + vy)$, we will be done (here, $v_p(x)$ denotes the p-adic valuation of $x$).

Suffice it to show that $v_{p_i}(\gcd(ax + cy, bx + dy)) = v_{p_i}(x + Vy)$ for all $p_i$. Using the fact that $p_i | D = bc - ad$ and $\gcd(a, b) = \gcd(c, d) = 1$, $p_i$ divides either zero, one, or two of $a, b, c, d$. However, $p_i$ cannot divide just one of $a, b, c, d$ lest it divide exactly one of $ad, bc$ contradicting the fact that it divides $D$, their difference.  We consider two cases:

{\bf Case 1}: $p_i \nmid a, b, c, d$.

Since $a$ and $b$ are invertible $\mod{p_i^{e_i}}$, and $v_{p_i}(ax + cy), v_{p_i}(bx + dy) \le v_{p_i}(D)$, $$v_{p_i}(bx + dy) = v_{p_i}(a(bx + dy)) = v_{p_i}(abx + ady + Dy) = v_{p_i}(abx + cby) = v_{p_i}(ax + cy)$$

Let $a^{-1}$ denote the inverse of $a$, $\mod{p_i^{e_i}}$. Since of course $p_i \nmid a^{-1}$, $v_{p_i}(ax + cy) = v_{p_i}(a^{-1}(ax + cy)) = v_{p_i}(x + a^{-1}cy)$.  Then if we take $V \equiv a^{-1}c \pmod{p_i^{e_i}}$, we are guaranteed that $v_{p_i}(\gcd(ax + cy, bx + dy)) = v_{p_i}(x + Vy)$.

{\bf Case 2}: Exactly two of $a, b, c, d$ are divisible by $p_i$.

Notice that $p_i$ cannot divide $a$ and $b$ simultaneously, lest $\frac{a}{b}$ would not be reduced. Similarly, $p_i$ cannot divide $c$ and $d$ simultaneously. It is also not possible that $p_i | a, d$ or $p_i | b, c$ since otherwise exactly one of $bc, ad$ would be divisible by $p_i$ contradicting the fact that their difference $D$ is divisible by $p_i$. Then either $p_i | a, c$ or $p_i | b, d$.

Without loss of generality suppose $p_i | a, c$. Then $b$ and $d$ are invertible $\pmod{p_i^{e_i}}$ so we can write $$v_{p_i}(bx + dy) = v_{p_i}(b^{-1}(bx + dy)) = v_{p_i}(x + b^{-1}d).$$ If we can show that $v_{p_i}(ax + cy) \ge v_{p_i}(bx + dy)$, we can simply take $V \equiv b^{-1}d \pmod{p_i^{e_i}}$ to finish.

Now consider $v_{p_i}(a)$ and $v_{p_i}(c)$. If they are unequal, the lesser of the two is $e_i$, since otherwise $v_{p_i}(D) > e_i$ or $v_{p_i}(D) < e_i$, contradiction. Then $v_{p_i}(ax + cy) \ge e_i$ so $v_{p_i}(ax + cy) \ge v_{p_i}(bx + dy)$ as desired.

Otherwise, $v_{p_i}(a) = v_{p_i}(c) = u_i \le e_i$. If $u_i = e_i$, we can make the same argument as above and $v_{p_i}(ax + cy) \ge v_{p_i}(bx + dy)$. If not, $a = p_i^{u_i}a'$ and $c = p_i^{u_i}c'$ with $p_1 \nmid a', c'$. Additionally, $v_{p_i}(a'd - bc') = e_i - u_i$. 

Suppose $p_i^{w_i} || bx + dy$. To show that $v_{p_i}(bx + dy) \ge v_{p_i}(ax + cy)$, it is enough to prove that $p^{w_i - u_i} | a'x + c'y$. Notice that $$(a'd - bc')x = d(a'x + c'y) - c'(bx + dy),$$ where $p_i^{e_i - u_i} | a'd - bc'$ and $p_i^{w_i} | bx + dy$. Since $p_i \nmid d$, taking the equation $\mod{p_i^{\min(e_i - u_i, w_i)}}$ gives $v_{p_i}(bx + dy) \ge \min(e_i - u_i, w_i)$. Of course $w_i > w_i - u_i$ but also $e_i \ge w_i$ so $e_i - u_i > w_i - u_i$. Therefore, $v_{p_i}(bx + dy) \ge v_{p_i}(ax + cy)$, so we can choose $V$ as we did above and we are done.

\end{proof}

We can take advantage of the linearity of the equivalent tree $T(\frac{0}{1}, \frac{D}{V})$ to prove the following:

\begin{theorem}
If $\frac{a}{b}, \frac{c}{d}$ are any two rational numbers, $T(\frac{a}{b}, \frac{c}{d})$ contains all rational numbers in the interval $[\frac{a}{b}, \frac{c}{d}].$
\end{theorem}

\begin{proof}
We prove by strong induction on the value of the cross-determinant $D = bc - ad$.
When $D = 1$, the result is just Theorem 3.1. Now suppose the result holds for all values $D \le n$ for some $n$. To show that it holds for $D = n + 1$, we will show that for all positive integers $V$, the tree $T(\frac{0}{1}, \frac{n + 1}{V})$ contains all rational numbers in the interval $[\frac{0}{1}, \frac{n + 1}{V}]$. By Lemma 3.5 and Theorem 3.6, this is sufficient.

Consider any rational number $x \in [\frac{0}{1}, \frac{n + 1}{V}]$. We know the rational number $\frac{x}{n + 1}$ appears in $T(\frac{0}{1}, \frac{1}{V})$ since this tree has cross-determinant $1$. Suppose this fraction appears for the first time in row $k$. For $0 \le i < k$, define $L_i$ to be the greatest fraction less than $\frac{x}{n + 1}$ in row $i$ of $T(\frac{0}{1}, \frac{1}{V})$. Similarly, let $R_i$ be the least fraction greater than $\frac{x}{n + 1}$ in row $i$ of $T(\frac{0}{1}, \frac{1}{V})$.

Suppose for the sake of contradiction that $x$ does not appear in $T(\frac{0}{1}, \frac{D}{V})$. Let us analogously define $l_i$ to be the greatest fraction less than $x$ in row $i$ of $T(\frac{0}{1}, \frac{D}{V})$ and $r_i$ to be the least fraction greater than $x$. 

First, $l_0 = \frac{0}{1}$ and $r_0 = \frac{n + 1}{V}$ while $L_0 = \frac{0}{1}$ and $R_0 = \frac{1}{V}$. We also know that $l_{i + 1}$ and $r_{i + 1}$ are either $l_i$ and the mediant of $l_i, r_i$ or the mediant of $l_i, r_i$ and $r_i$. If this mediant is ever reduced, the determinant of $l_{i + 1}$ and $r_{i + 1}$ is reduced by the same factor meaning it is strictly less than $D$, the determinant of $l_i$ and $r_i$. But since $x \in [l_{i + 1}, r_{i + 1}]$ the inductive hypothesis means $x \in T(l_i, r_i)$ which is itself contained in $T(\frac{0}{1}, \frac{n + 1}{V})$. 
 
Then suppose the mediant of $l_i, r_i$ never needs to be reduced for any $i$. By the linearity of addition, a simple inductive argument gives $l_i = (n + 1)L_i$ and $r_i = (n + 1)R_i$. Since the mediant of $L_{k - 1}$ and $R_{k - 1}$ is $\frac{x}{n + 1}$ by hypothesis, and because the mediant of $l_{k - 1}$ and $r_{k - 1}$ does not have to be reduced, the mediant fraction formed from $l_{k - 1}$ and $r_{k - 1}$ must be $(n + 1)(\frac{x}{n + 1}) = x$, contradicting the fact that $x$ does not appear in $T(\frac{0}{1}, \frac{n + 1}{V})$.

Then by induction, every rational number in $[\frac{a}{b}, \frac{c}{d}]$ appears in $T(\frac{a}{b}, \frac{c}{d})$ regardless of the choice of $a, b, c, d$.

\end{proof}

\section{Acknowledgements}

We thank the PRIMES program at MIT for allowing us the opportunity to do this project. We also thank Dr. Tanya Khovanova (MIT) for helping guide the research, teaching the author to program in Mathematica, and greatly improving the quality of writing in the paper. Finally, we would like to thank Prof. James Propp (UMass) for suggesting the project and discussing it with the author.

\end{document}